\documentclass[4pt, oneside]{article}

\usepackage[english]{babel}
\usepackage{setspace}
\usepackage{graphicx, caption, color}
\usepackage{amsthm,amsmath,amssymb,stmaryrd,bm,upgreek,multicol,nicefrac}
\usepackage{enumerate}
\usepackage[inline]{enumitem}
\usepackage{bussproofs}
\usepackage[colorlinks]{hyperref}
\hypersetup{
  colorlinks,
  citecolor={black},
  linkcolor={red},
  urlcolor={blue}}

\makeatletter
\@ifpackageloaded{stix}{%
}{%
  \DeclareFontEncoding{LS2}{}{\noaccents@}
  \DeclareFontSubstitution{LS2}{stix}{m}{n}
  \DeclareSymbolFont{stix@largesymbols}{LS2}{stixex}{m}{n}
  \SetSymbolFont{stix@largesymbols}{bold}{LS2}{stixex}{b}{n}
  \DeclareMathDelimiter{\lBrace}{\mathopen} {stix@largesymbols}{"E8}%
                                            {stix@largesymbols}{"0E}
  \DeclareMathDelimiter{\rBrace}{\mathclose}{stix@largesymbols}{"E9}%
                                            {stix@largesymbols}{"0F}
}
\makeatother

\newcommand{\la}{\langle \,}
\newcommand{\ra}{\, \rangle}
\newcommand{\al}{\alpha}
\newcommand{\be}{\beta}

\newcommand{\Lim}{\text{Lim}}
\newcommand{\ob}{\texttt{\textup (} }
\newcommand{\cb}{ \texttt{\textup )}}
\newcommand{\BC}{\textbf{\textup{BC}}}
\newcommand{\RC}{\bm{\mathbf{RC}_{\Gamma_0}}}
\newcommand{\GLP}{\mathbf{GLP}}
\newcommand{\GLPG}{\bm{\mathbf{GLP}_{\Gamma_0}}}

\newtheorem{Theorem}{Theorem}
\newtheorem{Proposition}[Theorem]{Proposition}
\newtheorem{Example}[Theorem]{Example}

\newtheorem{Definition}[Theorem]{Definition}
\newtheorem{Lemma}[Theorem]{Lemma}
\newtheorem{Corollary}[Theorem]{Corollary}
\newtheorem{Remark}[Theorem]{Remark}

\newcommand{\xcases}[2]{\nicefrac #1  #2}

\newcommand{\ignore}[1]{}

\newcommand{\worms}[1]{{\mathbb W}_{#1}}

\newcommand{\wormsBC}{\worms{\ob\cb}}
\newcommand{\forms}[1]{\mathbb F_{#1}}

\newcommand{\form}{\forms{\ob\cb}}

\newcommand{\ord}{{\sf Ord}}

\newcommand{\fv}{B}
\newcommand{\fw}{A}

\newcommand{\wge}[1]{  >_0}
\newcommand{\wgeq}[1]{  \geq_0}
\newcommand{\wle}[1]{  <_0}
\newcommand{\wleq}[1]{  \leq_0}

\newcommand{\fs}[2]{#1 [ #2 ] }
\newcommand{\fsi}[2]{#1\llbracket #2 \rrbracket }
\newcommand{\ga}{\gamma}
\newcommand{\vnf}{\equiv_{\sf VNF}}
\newcommand{\ot}{\upphi}

\newcommand{\atr}{\mathbf{ATR}_0}

\begin{document}

\title{Deducibility and Independence in Beklemishev's Autonomous Provability Calculus}

\author{David Fern\'andez-Duque\footnote{\href{mailto:david.fernandezduque@ugent.be}{david.fernandezduque@ugent.be}}  \\ Eduardo Hermo Reyes\footnote{\href{mailto:ehermo.reyes@ub.edu}{ehermo.reyes@ub.edu}}} 
\maketitle

\begin{abstract}
Beklemishev introduced an ordinal notation system for the Feferman-Sch\"utte ordinal $\Gamma_0$ based on the \emph{autonomous expansion} of provability algebras. In this paper we present the logic $\BC$ (for \emph{Bracket Calculus}). The language of $\BC$ extends said ordinal notation system to a strictly positive modal language. Thus, unlike other provability logics, $\BC$ is based on a self-contained signature that gives rise to an ordinal notation system instead of modalities indexed by some ordinal given {\em a priori.} The presented logic is proven to be equivalent to $\RC$, that is, to the strictly positive fragment of $\GLPG$.
We then define a combinatorial statement based on $\BC$ and show it to be independent of the theory $\atr$ of Arithmetical Transfinite Recursion, a theory of second order arithmetic far more powerful than Peano Arithmetic.

\end{abstract}

\section{Introduction}

In view of G\"odel's second incompleteness theorem, we know that the consistency of any sufficiently powerful formal theory cannot be established using purely `finitary' means.
Since then, the field of proof theory, and more specifically of ordinal analysis, has been successful in measuring the non-finitary assumptions required to prove consistency assertions via computable ordinals.
Among the benefits of this work is the ability to linearly order natural theories of arithmetic with respect to notions such as their `consistency strength' (e.g., their $\Pi^0_1$ ordinal) or their `computational strength' (their $\Pi^0_2$ ordinal).
Nevertheless, the assignment of these proof-theoretic ordinals to formal theories depends on a choice of a `natural' presentation for such ordinals, with well-known pathological examples having been presented by Kreisel \cite{Kreisel1976} and Beklemishev \cite{BeklemishevPathological}.\footnote{The $\Pi^1_1$ ordinal of a theory is another measure of its strength and does not have such sensitivity to a choice of notation system. However, there are some advantages to considering $\Pi^0_1$ ordinals, among others that they give a finer-grained classification of theories.}
This raises the question of what it means for something to be a natural ordinal notation system, or even if such a notion is meaningful at all.

One possible approach to this problem comes from Beklemishev's ordinal analysis of Peano arithmetic ($\mathbf{PA}$) and related theories via their {\em provability algebras.}
Consider the Lindenbaum algebra of the language of arithmetic modulo provability in a finitary theory $U$ such as primitive recursive arithmetic ($\mathbf{PRA}$) or the weaker {\em elementary arithmetic} ($\mathbf{EA}$).
For each natural number $n$ and each formula $\varphi$, the $n$-consistency of $\varphi$ is the statement that all $\Sigma_n$ consequences of $U + \varphi$ are true, formalizable by some arithmetical formula $\langle n \rangle \varphi$ (where $\varphi$ is identified with its G\"odel number).
In particular, $\langle 0 \rangle \varphi$ states that $\varphi$ is consistent with $U$.
An {\em iterated consistency assertion,} also called {\em worm,} is then an expression of the form $\langle n_1 \rangle \hdots \langle n_k\rangle \top$, where $\top$ is some fixed tautology.

The operators $\langle n\rangle$ and their duals $[n]$ satisfy Japaridze's provability logic $\mathbf{GLP}$ \cite{Japaridze:1986:PhdThesis}, a multi-modal extension of the G\"odel-L\"ob provability logic $\mathbf{GL}$ \cite{Boolos:1993:LogicOfProvability}.
As Beklemishev showed, the set of worms is well-ordered by their {\em consistency strength} $\wle{}$, where $\fw \wle{} \fv$ if $\fw \to \langle 0 \rangle \fv$ is derivable in $\mathbf{GLP}$.
Moreover, this well-order is of order-type $\varepsilon_0$, which characterizes the proof-theoretical strength of $\mathbf{PA}$.
This tells us that proof-theoretic ordinals already appear naturally within Lindenbaum algebras of arithmetical theories.
Using these ideas, Beklemishev has shown how the logic $\mathbf{GLP}$ gives rise to the {\em Worm principle,} a relatively simple combinatorial principle which is independent of $\mathbf{PA}$ \cite{Beklemishev:2004:ProvabilityAlgebrasAndOrdinals}. 

Beklemishev also observed that this process can be extended by considering worms with ordinal entries.
Extensions of $\GLP$, denoted ${\GLP}_\Lambda$, have been considered in cases where $\Lambda$ is an ordinal \cite{Beklemishev:2005:VeblenInGLP,Fernandez:2012:TopologicalCompleteness,FernandezJoosten:2012:WellOrders} or even an arbitrary linear order \cite{BeklemishevFernandezJoosten:2012:LinearlyOrderedGLP}.
Proof-theoretic interpretations for $\mathbf{GLP}_\Lambda$ have been developed by Fern\'andez-Duque and Joosten \cite{FernandezJoosten:2013:OmegaRuleInterpretationGLP} for the case where $\Lambda$ is a computable well-order.
Nevertheless, we now find ourselves in a situation where an expression $\langle \lambda \rangle \varphi$ requires a system of notation for the ordinal $\lambda$.
Fortunately we may `borrow' this notation from finitary worms and represent $\lambda$ itself as a worm.
Iterating this process we obtain the {\em autonomous worms,} whose order types are exactly the ordinals below the Feferman-Sch\"utte ordinal $\Gamma_0$.
By iterating this process we obtain a notation system for worms which uses only parentheses, as ordinals (including natural numbers) can be iteratively represented in this fashion. Thus the worm $\langle 0 \rangle \top$ becomes $\ob\cb$, $\langle 1 \rangle \top$ becomes $\ob\ob\cb\cb$, $\langle \omega \rangle \top$ becomes $\ob\ob\ob\cb\cb\cb$, etc.

These are Beklemishev's {\em brackets,} which provide a notation system for $\Gamma_0$ without any reference to an externally given ordinal \cite{Beklemishev:2005:VeblenInGLP}.
However, it has the drawback that the actual computation of the ordering between different worms is achieved via a translation into a traditional ordinal notation system.
We will remove the need for such an intermediate step by providing an autonomous calculus for determining the ordering relation (and, more generally, the logical consequence relation) between bracket notations.
To this end we present the {\em Bracket Calculus} ($\mathbf{BC}$).
We show that our calculus is sound and complete with respect to the intended embedding into $\GLP_{\Gamma_0}$.
We then show that the Worm principle can be naturally extended to $\mathbf{BC}$ to yield independence for theories of strength $\Gamma_0$, particularly the theory Arithmetical Transfinite Recursion $\atr$, one of the `Big Five' of reverse mathematics \cite{Simpson:2009:SubsystemsOfSecondOrderArithmetic}.

\section{The Reflection Calculus}

Japaridze's logic $\GLP$ gained much interest due to Beklemishev's proof-theoretic applications \cite{Beklemishev:2004:ProvabilityAlgebrasAndOrdinals}; however, from a modal logic point of view, it is not an easy system to work with. To this end, in \cite{Beklemishev:2012:CalibratingProvabilityLogic,Beklemishev:2013:PositiveProvabilityLogic,Dashkov:2012:PositiveFragment} Beklemishev and Dashkov introduced the system called \emph{Reflection Calculus}, \textbf{RC}, that axiomatizes the fragment of $\GLP_\omega$ consisting of implications of strictly positive formulas. This system is much simpler than $\GLP_\omega$ but yet expressive enough to maintain its main proof-theoretic applications.
In this paper we will focus exclusively on reflection calculi, but the interested reader may find more information on the full $\GLP$ in the references provided.

Similar to $\GLP_\Lambda$, the signature of $\textbf{RC}_\Lambda$ contains modalities of the form $\la \al \ra$ for $\al \in \Lambda$. However, since this system only considers strictly positive formulas, the signature does not contain negation, disjunction or modalities $\lbrack \, \al \, \rbrack$. Thus, consider a modal language $\mathcal{L}$ with a constant $\top$, a set of propositional variables $p,q,\ldots$, a binary connective $\wedge$ and unary connectives $\la \al \ra$, for each $\al \in \Lambda$. The set of formulas in this signature is defined as follows:

\begin{Definition}
Fix an ordinal $\Lambda$.
By $\mathbb{F}_{\Lambda}$ we denote the set of formulas built up by the following grammar:
\[
\varphi := \top \ | \ p \ | \ (\varphi \land \psi) \ | \ \la \al \ra \varphi \ \ \text{ for } \al \in \Lambda.
\]
\end{Definition}

Next we define a consequence relation over $\mathbb{F}_{\Lambda}$.
For the purposes of this paper, a {\em deductive calculus} is a pair $\mathbf X = ( \mathbb F_\mathbf X, \vdash_\mathbf X)$ such that $\mathbb F_\mathbf X$ is some set, the {\em language} of $\mathbf X$, and ${\vdash_\mathbf X} \subseteq \mathbb F_\mathbf X \times \mathbb F_\mathbf X$.
We write $\varphi \equiv_\mathbf X \psi$ for $\varphi \vdash_\mathbf X \psi$ and $\psi \vdash_\mathbf X \varphi$.
We will omit the subscript $\mathbf X$ when this does not lead to confusion, including in the definition below, where $\vdash$ denotes $\vdash_{\mathbf{RC}_\Lambda}$.

\begin{Definition}
Given an ordinal $\Lambda$, the calculus $\mathbf{RC}_\Lambda$ over $\forms\Lambda$ is given by the following set of axioms and rules:\medskip

\noindent Axioms:
\begin{multicols}2
\begin{enumerate}
\item \label{RCax}
$\varphi \vdash \varphi, \ \ \ \varphi \vdash \top$;	

\item 
$\varphi \wedge \psi \vdash \varphi,  \ \ \ \varphi \wedge \psi \vdash \psi$;	\label{RCconjel}

\item 
$\la \al \ra \la \al \ra \varphi \vdash \la \al \ra \varphi$; 
\label{RCtrans}

\item 
$\la \al \ra \varphi \vdash \la \be \ra \varphi$ \ \ \ for $\al > \be$;
\label{RCmon}

\item
$\la \al \ra \varphi  \land \la \be\ra \psi \vdash \la \al \ra  \big( \varphi \land \la \be \ra \psi\big)$ \ \ \ for $\al > \be$.
\label{RCconj}
\end{enumerate}
\end{multicols}

\noindent Rules:
\begin{multicols}2
\begin{enumerate}
\item \label{Rr:1}
If $\varphi \vdash \psi$ and $\varphi\vdash\chi$, then $\varphi\vdash \psi \wedge \chi $;	

\item \label{Rr:2}
If $\varphi\vdash \psi$ and $\psi \vdash \chi$, then $\varphi\vdash \chi $;	

\item \label{Rr:3}
If $\varphi\vdash \psi$, then $\la \al \ra \varphi \vdash \la \al \ra \psi$;	
\end{enumerate}
\end{multicols}
\end{Definition}

For each $\textbf{RC}_\Lambda$-formula $\varphi$, we can define the signature of $\varphi$ as the set of ordinals occurring in any of its modalities.
\begin{Definition}
For any $\varphi \in \mathbb{F}_\Lambda$, we define the \emph{signature} of $\varphi$, $\mathcal{S}(\varphi)$, as follows:\\
\begin{enumerate*}
\item $\mathcal{S}(\top) = \mathcal{S}(p) = \varnothing$;\\
\item $\mathcal{S}(\varphi \land \psi) = \mathcal{S}(\varphi) \cup \mathcal{S}(\psi)$;\\
\item $\mathcal{S}(\la \al \ra \varphi) = \{\al\} \cup \mathcal{S}(\varphi)$. 
\end{enumerate*} 
\end{Definition}

With the help of this last definition we can make the following observation: 

\begin{Lemma} \label{Lemma:OrdinalsBoundedByDerivability}
For any $\varphi, \, \psi \in \mathbb{F}_\Lambda$:\\
\begin{enumerate*}
\item If $\mathcal{S}(\psi) \neq \varnothing$ and $\varphi \vdash \psi$, then $\max\mathcal{S}(\varphi) \geq \max\mathcal{S}(\psi)$; \label{Ord>0}\\

\item If $\mathcal{S}(\varphi) = \varnothing$ and $\varphi \vdash \psi$, then $\mathcal{S}(\psi) = \varnothing$. \label{Ord=0}
\end{enumerate*} 
\end{Lemma}
\begin{proof}
By an easy induction on the length of the derivation of $\varphi \vdash \psi$.
\end{proof}

The reflection calculus has natural arithmetical \cite{FernandezJoosten:2013:OmegaRuleInterpretationGLP}, Kripke \cite{Dashkov:2012:PositiveFragment,Beklemishev:2013:PositiveProvabilityLogic}, algebraic \cite{BeklemishevUniversal} and topological \cite{BeklemishevGabelaia:2011:TopologicalCompletenessGLP,Fernandez:2012:TopologicalCompleteness,Icard:2009:TopologyGLP,Ignatiev:1993:StrongProvabilityPredicates} interpretations for which it is sound and complete, but in this paper we will work exclusively with reflection calculi from a syntactical perspective.
Other variants of the reflection calculus have been proposed, for example working exclusively with worms \cite{WormCalculus}, admitting the transfinite iteration of modalities \cite{HermoJoosten:2017:RelationalSemanticsTSC}, or allowing additional conservativity operators \cite{BeklemishevPartial,BeklemishevSpectra}.

\section{Worms and the consistency ordering}

In this section we review the consistency ordering between worms, along with some of their basic properties.

\begin{Definition}
Fix an ordinal $\Lambda$. The set of \emph{worms} in $\forms\Lambda$, $\worms\Lambda$, is recursively defined as follows:
\begin{enumerate*}
\item $\top \in \worms\Lambda$;
\item If $A \in \worms\Lambda$ and $\alpha < \Lambda$, then $\la \al \ra A \in \worms\Lambda$. 
\end{enumerate*}
Similarly, we inductively define for each $\al \in \Lambda$ the set of worms $\worms\Lambda^{\geq \al}$ where all ordinals are at least $\al$: 
\begin{enumerate*}
\item $\top \in \worms\Lambda^{\geq \al}$;
\item If $A \in \worms\Lambda^{\geq \al}$ and $\be \geq \al$, then $\la \be \ra  A \in \worms\Lambda^{\geq \al}$.
\end{enumerate*}
\end{Definition}

From now on, we shall make use of the usual arithmetical operations on ordinal numbers such as addition ($\al + \be$) or exponentiation ($\al^\be$); see e.g.~\cite{Spiders} for definitions.

\begin{Definition}\label{defUparrow}
Let ${A}= \langle \xi_1 \rangle \hdots \langle \xi_n \rangle \top$ and $B= \langle \zeta_1 \rangle \hdots \langle \zeta_m \rangle \top$ be worms. Then, define
${A} B= \langle \xi_1 \rangle \hdots \langle \xi_n \rangle \langle  \zeta_1 \rangle \hdots \langle \zeta_m \rangle \top.$
Given an ordinal $\lambda$, define $\lambda \uparrow A$ to be
$\langle \lambda + \xi_1 \rangle \hdots \langle \lambda + \xi_n \rangle \top.$
\end{Definition}

Often we will want to put an extra ordinal between two worms, and we write ${B} \langle \lambda \rangle A$ for $\fv(\langle \lambda \rangle \fw)$.
Next, we define the consistency ordering between worms.

\begin{Definition}
Given an ordinal $\Lambda$, we define a relation $\wle{\lambda}$ on $\worms\Lambda$ by  ${B} \wle {\lambda} A$ if and only if $ A \vdash \langle 0 \rangle {B}.$ We also define $\fv\wleq\mu\fw$ if $\fv\wle\mu\fw$ or $\fv\equiv\fw$.
\end{Definition}

The ordering $\wleq{}$ has some nice properties.
Recall that if $A$ is a set (or class), a {\em preorder} on $A$ is a transitive, reflexive relation ${\preccurlyeq}\subseteq A\times A$. The preorder $\preccurlyeq$ is {\em total} if, given $a,b\in A$, we always have that $a\preccurlyeq b$ or $b\preccurlyeq a$, and {\em antisymmetric} if whenever $a\preccurlyeq b$ and $b\preccurlyeq a$, it follows that $a=b$. A total, antisymmetric preorder is a {\em linear order.} We say that $\langle A,\preccurlyeq\rangle$ is a {\em pre-well-order} if $\preccurlyeq$ is a total preorder and every non-empty $B\subseteq A$ has a minimal element (i.e., there is $m\in B$ such that $m\preccurlyeq b$ for all $b\in B$). A {\em well-order} is a pre-well-order that is also linear. Note that pre-well-orders are not the same as well-quasiorders (the latter need not be total).
Pre-well-orders will be convenient to us because, as we will see, worms are pre-well-ordered but not linearly ordered. 
The following was first proven by Beklemishev \cite{Beklemishev:2005:VeblenInGLP}, and a variant closer to our presentation may be found in \cite{Spiders}.

\begin{Theorem}\label{TheoWormsWO}
For any ordinal $\Lambda$, the relation $\wleq\eta$ is a pre-well-order on $\worms \Lambda$.
\end{Theorem}

This yields as a corollary a nice characterization of $\wleq {}$.

\begin{Corollary}\label{corLeq}
Given an ordinal $\Lambda$ and $A,B\in \worms\Lambda$, $A\wgeq {}B$ if and only if $A\vdash _{\RC} B$ or $A\vdash _{\RC} \la 0\ra B$.
\end{Corollary}

\begin{proof}
By definition, $A\wgeq {}B$ if and only if $A\equiv _{\RC} B$ or $A\vdash _{\RC} \la 0\ra B$, so it remains to prove that $A\vdash _{\RC} B$ implies that $A\wgeq {}B$.
Otherwise, since $\wleq{}$ is a pre-well-order, $A\wle{}B$, so that $B \vdash _{\RC} \la 0\ra A  \vdash _{\RC} \la 0\ra B$, yielding $B \wle{} B$, and violating the well-foundedness of $\wle{}$.
\end{proof}

Note that $\wleq\eta$ fails to be a linear order merely because it is not antisymmetric.
To get around this, one may instead consider worms modulo provable equivalence.
Alternatively, as Beklemishev has done \cite{Beklemishev:2005:VeblenInGLP}, one can choose a canonical representative for each equivalence class.

\begin{Definition}[Beklemishev Normal Form] A worm $A \in \worms\Lambda$ is defined recursively to be in ${\sf BNF}$ if either
\begin{enumerate}
\item $A = \top$, or
\item $A := A_k \la \al \ra A_{k-1} \la \al \ra \ldots \langle \al \rangle A_0$ with
\begin{itemize}
\item $\al = \min \mathcal{S}(A)$;
\item $k \geq 1$;
\item $A_i \in \mathbb{W}_\Lambda^{\geq \al+1}$, for $i \leq k$;
\end{itemize}
such that $A_i \in {\sf BNF}$ and $A_i \vdash_{\RC} \la \al + 1  \ra A_{i+1}$ for each $i < k$.
\end{enumerate}

\end{Definition}

This definition essentially mirrors that of Cantor normal forms for ordinals.
The following was proven in \cite{Beklemishev:2005:VeblenInGLP}.

\begin{Theorem}\label{TheoBNF}
Given any worm $\fw$ there is a unique $\fw'\in \sf BNF$ such that $\fw \equiv_{\RC} \fw'$.
\end{Theorem}

It follows immediately that for any ordinal $\Lambda$, $(\worms \Lambda \cap {\sf BNF},\wleq{})$ is a well-order.

\section{Hyperexponential notation for $\Gamma_0$}\label{secHyper}

{Ordinal numbers} are canonical representatives of well-orders; we assume some basic familiarity with them, but a detailed account can be found in a text such as \cite{Jech:2002:SetTheory}.
In particular, since the set of worms modulo equivalence yields a well-order, we can use ordinal numbers to measure their order-types.
More generally, if $\mathfrak A=\langle A,\preccurlyeq\rangle$ is any pre-well-order, for $a\in A$ we may define a function $o\colon A\to \ord$ given recursively by
$o(a)=\sup_{b\prec a}(o(b)+1)$, where by convention $\sup \varnothing = 0$,
representing the {\em order-type} of $a$; this definition is sound since $\mathfrak A$ is pre-well-ordered.
The rank of $\mathfrak A$ is then defined as $\sup_{a\in A} (o(a) + 1)$.

The following lemma is useful in characterizing the rank function \cite{Spiders}.

\begin{Lemma}\label{LemmOUnique}

Let $\langle A,\preccurlyeq \rangle$ be a well-order. Then $o\colon A\to\ord$ is the unique function such that

\begin{enumerate}

\item $x\prec y$ implies that $o(x)<o(y)$,

\item\label{itOUniqueTwo} if $\xi < o(x)$ then $\xi = o(y)$ for some $y\in A$.

\end{enumerate}

\end{Lemma}

In order to compute the ordinals $o(\fw)$, let us recall a notation system for $\Gamma_0$ using {\em hyperexponentials} \cite{FernandezJoosten:2012:Hyperations}.
The class of all ordinals will be denoted $\ord$, and $\omega$ denotes the first infinite ordinal. Recall that many number-theoretic operations such as addition, multiplication and exponentiation can be defined on the class of ordinals by transfinite recursion.
The ordinal exponential function $\xi \mapsto \omega ^\xi$ is of particular importance for representing ordinal numbers.
When working with order types derived from reflection calculi, it is convenient to work with a slight variation of this exponential.
\begin{Definition}[Exponential function]
The \emph{exponential function} is the function $e\colon \ord \to \ord$ given by \mbox{$\xi \mapsto -1 + \omega^\xi$}.
\end{Definition}

Observe that for $\xi = 0$, we have that $e\, \xi = -1 + \omega^0 = -1 + 1 =0$. The function $e$ is an example of a {\em normal function,} i.e.~$f\colon \ord \to \ord$ which is strictly increasing and {\em continuous,} in the sense that if $\lambda$ is a limit then $f(\lambda) = \sup_{\xi<\lambda} f(\xi)$.
Giving a mapping $f\colon X\to X$, it is natural and often useful to ask whether $f$ has {\em fixed points,} i.e., solutions to the equation $x=f(x)$. In particular, normal functions have many fixed points.

\begin{Proposition}\label{fixedpoints}
Every normal function $f\colon \ord\to\ord$ has arbitrarily large fixed points.
The least fixed point of $f$ greater or equal than $\alpha$ is given by $\lim_{n\to\infty} f^n \alpha $.
\end{Proposition}

The first ordinal $\alpha$ such that $\alpha = \omega^\alpha$ is the limit of the $\omega$-sequence\linebreak $(\omega, \omega^\omega, \omega^{\omega^\omega}, \hdots)$, and is usually denoted $\varepsilon_0$.
Every $\xi <\varepsilon_0$ can be written in terms of $0$ using only addition and the function $\omega \mapsto \omega^\xi$ via its {Cantor normal form.}
The hyperexponential function is then a natural transfinite iteration of the ordinal exponential which remains normal after each iteration.

\begin{Definition}[Hyperexponential functions]
The \emph{hyperexponential functions}  $(e^\zeta)_{\zeta \in \ord}$ are the unique family of normal functions that satisfy
\begin{enumerate}
\item $e^0 = \textup{id}$,
\item\label{DefHypI} $e^1 = e$,
\item\label{DefHypII} $e^{\alpha + \beta} = e^\alpha \circ e^\beta$ for all $\alpha$ and $\beta$, and
\item if $(f^\xi)_{\xi\in\ord}$ is a family of functions satisfying \ref{DefHypI} and \ref{DefHypII}, then for all $\alpha,\beta\in\ord$, $e^\alpha\beta\leq f^\alpha\beta$.
\end{enumerate}
\end{Definition}

Fern\'andez-Duque and Joosten proved that the hyperexponentials are well-defined \cite{FernandezJoosten:2012:Hyperations}.
If $\alpha>0$ then $e^\alpha \beta$ is always {\em additively indecomposable} in the sense that $\xi,\zeta < e^\alpha \beta$ implies that $\xi + \zeta < e^\alpha \beta$; note that zero is additively indecomposable according to our definition.
In \cite{Spiders} it is also shown that the function $\xi \mapsto e^\xi 1$ is itself a normal function, hence it has a least non-zero fixed point: this fixed point is the Feferman-Sch\"utte ordinal, $\Gamma_0$.
Just like ordinals below $\varepsilon_0$ may be written using $0$, addition, and $\omega$-exponentiation, every ordinal below $\Gamma_0$ may be written in terms of $0$, $1$, addition and the function $(\xi,\zeta) \mapsto e^\xi \zeta$.
The following was first proven in \cite{Beklemishev:2005:VeblenInGLP} with different notation, and in the current form in \cite{FernandezJoosten:2012:Hyperations}.

\begin{Theorem}\label{TheoTranOrder}
Let ${A},B $ be worms and $\alpha $ be an ordinal. Then,
\begin{enumerate}

\item $o(\top)=0$,

\item $o(B \la 0 \ra A) ={o(A)}+1+o({B}),$ and\label{TheoTranOrderItSum}

\item $o(\alpha\uparrow A) =e^\alpha {o({A})}.$\label{TheoTranOrderItThree}

\end{enumerate}
\end{Theorem}

While hyperexponential notation is more convenient for computing order types of worms, it can easily be translated back and forth into notation based on Veblen functions.
Given an ordinal $\alpha$, recall that $\upphi_\alpha$ is defined recursively so that $\upphi_0\beta:=\omega^\beta$ and for $\alpha>0$, $\upphi_\al\be$ is the $\beta$-th member of $\{\eta:(\forall\xi<\al) [\upphi_\xi \eta=\eta]\}$. Then, $\Gamma_0$ is the first non-zero ordinal closed under $(\alpha,\beta)\mapsto \upphi_\alpha\beta$.
We then have the following equivalences \cite[Proposition 5.15]{Spiders}:

\begin{Proposition}\label{propEvsPhi}
Given ordinals $\al, \be$:
\begin{enumerate}
\item $e^\al(0) = 0$;
\item $e^1(1 +\be) = \upphi_0(1 +\be)$;
\item $e^{\omega^{1+\al}}(1 +\be) =\upphi_{1+\al}(\be)$.
\end{enumerate}
\end{Proposition}

Finally we mention a useful property of $o$ proven in \cite{Spiders}.

\begin{Lemma}\label{LemmLowerBound}
Let $\fw\not=\top$ be a worm and $\mu$ an ordinal. Moreover, let $\al$ be the greatest ordinal appearing in $A$. Then, 
\begin{enumerate}

\item  if $\mu\leq\al$, then $ o( \langle \mu \rangle \top)\leq o(\fw)$, and\label{LemmLowerBoundItOne}

\item if $\al<\mu$, then $o(\fw)<o( \langle \mu \rangle \top)$.\label{LemmLowerBoundItTwo}

\end{enumerate}

\end{Lemma}

\section{Beklemishev's bracket notation system for $\Gamma_0$}

Before we introduce the full bracket calculus, let us review Beklemishev's notation system from \cite{Beklemishev:2005:VeblenInGLP}.

\begin{Definition}
By $\wormsBC$ we denote the smallest set such that:
\begin{enumerate*}
\item $\top \in \wormsBC$;
\item if $a, \, b \in \wormsBC$, then $\ob a \cb b \in \wormsBC$.
\end{enumerate*}
\end{Definition}
By convention we shall write $\ob \cb a$, for $a \in \wormsBC$, to the denote $\ob \top \cb a \in \wormsBC$. We shall also omit the use of $\top$ at the end of any worm $a\neq \top$, e.g., we shall use $\ob a_1 \cb \ldots \ob a_k \cb$ to denote $\ob a_1 \cb \ldots \ob a_k \cb \top$.
\medskip

We can define a translation $\ast: \wormsBC \to \worms\Lambda$ in such a way that an element $a \in \wormsBC$ will denote the ordinal $o(a^\ast)$:
\begin{enumerate}
\item $\top^\ast = \top$
\item $\big( \, \ob a \cb b \, \big)^\ast = \la o( a^\ast)\ra b^\ast$. 
\end{enumerate}
Therefore, we can also define $o^\ast : \wormsBC \to \ord$ as $o^\ast(a) = o(a^\ast)$.

\begin{Example}
We have that $o^\ast (\ob\ob\ob\cb\cb\cb) = \varepsilon_0$.
To see this, first note that $\ob\cb^\ast = \langle 0\rangle \top$, and $o(\langle 0\rangle \top) = 1$, where the calculation is performed using Theorem \ref{TheoTranOrder} in the `degenerate' case where $A=B=\top$.
It follows that $  \ob\ob\cb\cb^*= \langle o(\langle 0 \rangle \top) \rangle\top = \langle 1\rangle \top $, so that $o^* (\ob\ob\cb\cb) = o(\langle 1\rangle \top) = e^1 o(\langle 0 \rangle \top) = e^1 1 = \omega$.
By similar reasoning, $o^*(\ob\ob\ob\cb\cb\cb ) = o(\langle \omega\rangle \top) = e^\omega 1 = \upphi_1 0 =\varepsilon_0 $, where the second-to-last equality uses Proposition \ref{propEvsPhi} and the last is the definition of $\varepsilon_0$.
\end{Example}

In fact, $o^\ast : \wormsBC \to \Gamma_0$, and this map is surjective.
In order to prove this, we make some observations about how the ordinals represented by worms in $\wormsBC$ can be bounded in terms of the maximum number of nested brackets occurring in them. For this purpose, we introduce the following two definitions.

\begin{Definition}
For $a \in \wormsBC$, we define the nesting of $a$, ${\sf N}(a)$, as the maximum number of nested brackets. That is:
\begin{enumerate}
\item ${\sf N}(\top) = 0$;
\item ${\sf N}(\ob a \cb b) = \max\big({\sf N}(a) + 1, \, {\sf N}(b) \big)$. 
\end{enumerate} 
\end{Definition}

\begin{Definition}\label{defHFun}
We recursively define the function $h : \mathbb{N} \to \Gamma_0$ as follows:
\begin{enumerate}
\item $h(0) = 0$;
\item $h(n+1) = e^{h(n)} 1$.
\end{enumerate}
\end{Definition}

Note that $h$ is a strictly monotone function.
Using Proposition \ref{fixedpoints}, we see that $\lim_{n\to \infty} h(n) = \Gamma_0$.
In the following proposition we can find upper and lower bounds for any ordinal $o^\ast(a)$, with $a \in \wormsBC$, in terms of the nesting of $a$.

\begin{Proposition}
\label{proposition:nestedn}
For $a \in \wormsBC$, if ${\sf N}(a) = n $, then $h (n) \leq o^\ast(a) < h(n + 1)$.
\end{Proposition}
\begin{proof}
By induction on $n$. If $\bm{n=0}$ then we must have $a = \top$, hence $h(0) = 0 = o^\ast (a) < 1 = h(1) $.

For $\bm{n = n'+1}$, we have that $a = \ob a_0 \cb \ldots \ob a_m \cb$ for some $m \in \omega$. Moreover,
\begin{enumerate}
\item ${\sf N}(a_i) \leq n'$ for $i,\ 0 \leq i \leq m$;
\item there is $j \in \{0, \ldots, m\}$ such that ${\sf N}(a_J) = n'$. 
\end{enumerate}
Thus by the I.H. we get that $a^\ast = \langle \al_0 \rangle \ldots \langle \al_m \rangle \top$ such that:
\begin{enumerate}
\item For each $i$, $\al_i < h(n'+1)$;
\item there is $j \in \{0, \ldots, m\}$ such that $\al_j \geq h(n')$. 
\end{enumerate} 
By Lemma \ref{LemmLowerBound},
\[ o (\langle h(n') \rangle \top) \leq o(a^\ast) < o (\langle h(n'+1) \rangle \top); \]
but by Theorem \ref{TheoTranOrder} $ o (\langle h(n') \rangle \top) = e^{h(n')} 1 = h(n) $, while 
 $ o (\langle h(n'+1) \rangle \top) = e^{h(n)} 1 = h(n+1) $, as needed.
\end{proof}

As a consequence of this last proposition, we get the following corollaries.

\begin{Corollary}
For $a \in \wormsBC$, if ${\sf N}(a) = n $, then $a^\ast \in \mathbb{W}_{  h(n)}$.
\end{Corollary}
\begin{proof}
For ${\sf N}(a) = 0$, clearly we have that $a^\ast \in \mathbb{W}_{h(n)}$. For ${\sf N}(a) = n \, {>} \, 0$, let $a := \ob a_1 \cb \ldots \ob a_k \cb$. Thus, $a^\ast := \la o^\ast(a_1) \ra \ldots \la o^\ast(a_k) \ra \top$ where by Proposition \ref{proposition:nestedn} each $o^\ast(a_i) \, {<} \, h(n)$.
\end{proof}

\begin{Corollary} \label{Lemma:ordinalsboundnesting}
For $a, \, b \in \wormsBC$,
$
o^\ast(a) \geq o^\ast(b) \ \Rightarrow \ {\sf N}(a) \geq {\sf N}(b).
$
\end{Corollary}
\begin{proof}
We reason by contrapositive applying Proposition \ref{proposition:nestedn}. 
\end{proof}

\section{The Bracket Calculus}

In this section we introduce the \emph{Bracket Calculus}, denoted \BC. This system is analogous to $\RC$ and, as we will see later, both systems can be shown to be equivalent under a natural translation of $\BC$-formulas into $\RC$-formulas. 

The main feature of $\BC$ is that it is based on a signature that uses autonomous notations instead of modalities indexed by ordinals, whose ordering must be computed using a separate calculus. Moreover, since the order between these notations can be established in terms of derivability within the calculus, the inferences in this system can be carried out without using any external property of ordinals. In this sense, we say that $\BC$ provides an autonomous provability calculus.

The set of $\BC$-formulas, $\form$, is defined by extending $\wormsBC$ to a strictly positive signature containing a constant $\top$, a binary connective $\wedge$ and a set of propositional variables.

\begin{Definition}
By $\form$ we denote the set of formulas built up by the following grammar:
\[
\varphi := \top \ | \ p \ | \ \varphi \land \psi \ | \ \ob a \cb \, \varphi \ \ \text{ for } a \in \wormsBC.
\]
\end{Definition}

Similarly to \textbf{RC}, $\BC$ is based on \emph{sequents}, i.e. expressions of the form $\varphi \vdash \psi$, where $\varphi, \, \psi \in \form$. In addition to this, we will also use $b\mathrel \unlhd a$, for $a, \, b \in \wormsBC$, to denote that either $a \vdash \ob \cb \, b$ or $a \vdash b$ is derivable. Analogously, we will use $b \mathrel \lhd a$  to denote that the sequent $a \vdash \ob \cb b$ is derivable. 

\begin{Definition}
$\BC$ is given by the following set of axioms and rules:\\

\noindent Axioms:
\begin{enumerate*}
\item 
$\varphi \vdash \varphi, \ \ \ \varphi \vdash \top$;
\label{AX}	

\item 
$\varphi \wedge \psi \vdash \varphi,  \ \ \ \varphi \wedge \psi \vdash \psi$;	\label{conjel}

\end{enumerate*}
\

\noindent Rules:
\begin{enumerate}
\item \label{r:1}
If $\varphi \vdash \psi$ and $\varphi\vdash\chi$, then $\varphi\vdash \psi \wedge \chi $;	

\item \label{r:2}
If $\varphi\vdash \psi$ and $\psi \vdash \chi$, then $\varphi\vdash \chi $;	

\item \label{r:3}
If $\varphi\vdash \psi$ and $b \mathrel \unlhd a$, then $\ob a \cb \, \varphi \vdash \ob b \cb \, \psi$ and $\ob a \cb \, \ob b \cb \, \varphi \vdash \ob b \cb \, \psi$;

\item \label{r:5}
If $b\mathrel \lhd a$, then $\ob a \cb \, \varphi \land \ob b \cb \, \psi \vdash \ob a \cb \, \big( \varphi \land \ob b \cb \, \psi\big)$.
\end{enumerate}
\end{Definition}

\section{Translation and preservability}

In this section we introduce a way of interpreting $\BC$-formulas as $\RC$-formulas, and prove that under this translation, both systems can derive exactly the same sequents.

\begin{Definition}
We define a translation $\tau$ between $\form$ and $\mathbb{F}_{\Gamma_0}$, $\tau: \form \to \mathbb{F}_{\Gamma_0}$, as follows:
\begin{multicols}2
\begin{enumerate}
\item $\top^\tau = \top$;
\item $p^\tau = p$;
\item $(\varphi \land \psi)^\tau = (\varphi^\tau \land \psi^\tau)$;
\item $(\ob a \cb \, \varphi)^\tau =  \la o^\ast(a) \ra \varphi^\tau$.
\end{enumerate}
\end{multicols}
\end{Definition}

Note that for $a \in \wormsBC$, $a^\tau = a^\ast$.
Using this and routine induction, the following can readily be verified.

\begin{Lemma}\label{lemmStarInstance}
Given $\varphi \in \form$ and $\alpha \in \mathcal S(\varphi^\tau)$, there is a subformula $a \in \wormsBC$ of $\varphi$ such that $\alpha = o^* (a)$.
\end{Lemma}

The following lemma establishes the preservability of $\BC$ with respect to $\RC$, under $\tau$.

\begin{Lemma}\label{lemmBCtoRC}
For any $\varphi,\, \psi \in \form$:
$
\varphi \vdash_\BC \psi \ \Longrightarrow \ \varphi^\tau \vdash_{\RC} \psi^\tau.
$
\end{Lemma}
\begin{proof}
By induction on the length of the derivation. We can easily check that the set of axioms of $\BC$ is preserved under $\tau$. Likewise, the cases for a derivation ending on Rules \ref{r:1} or \ref{r:2} are straightforward. Thus, we only check Rules \ref{r:3} and \ref{r:5}. 

Regarding Rule \ref{r:3}, we need to prove that if $a \mathrel \unrhd b$ then both sequents $\langle o^\ast(a) \rangle \varphi^\tau \vdash \langle o^\ast(b) \rangle \psi^\tau$ and $\langle o^\ast(a) \rangle \langle o^\ast(b) \rangle \varphi^\tau \vdash \langle o^\ast(b) \rangle \psi^\tau$ are derivable in $\RC$. We can make the following observations by applying the I.H.:
\begin{enumerate}
\item Since $a \mathrel \unrhd b$, we have that either $a^\tau \vdash \langle 0 \rangle b^\tau$ or $a^\tau \vdash b^\tau$ are derivable in $\RC$. Therefore, by Corollary \ref{corLeq}, $o(a^\tau) \geq o(b^\tau)$. Since $o^*(a) = o(a^\ast) = o(a^\tau)$ and the same equality holds for $b$, we have that $o^*(a) \geq o^*(b)$.

\item We also have that $\varphi^\tau \vdash_{\RC} \psi^\tau$ and thus, by Rule \ref{Rr:3} of $\RC$ we obtain that $\langle o^\ast(a) \rangle \varphi^\tau \vdash \langle o^\ast(a) \rangle  \psi^\tau$ and $\langle o^\ast(a) \rangle \langle o^\ast(b) \rangle \varphi^\tau \vdash \langle o^\ast(a) \rangle \langle o^\ast(b) \rangle  \psi^\tau$ are derivable in $\RC$.
\end{enumerate}
On the one hand, by these two facts together with Axiom \ref{RCmon} we obtain that $\langle o^\ast(a) \rangle \varphi^\tau \vdash_{\RC} \langle o^\ast(b) \rangle \psi^\tau$. On the other hand, we can combine Axioms \ref{RCmon} and \ref{RCtrans} to get that $\langle o^\ast(a) \rangle \langle o^\ast(b) \rangle \varphi^\tau \vdash_{\RC} \langle o^\ast(b) \rangle \psi^\tau$.

We follow an analogous reasoning in the case of Rule \ref{r:5}. By the I.H. we have that $a^\tau \vdash_{\RC} \langle 0 \rangle b^\tau$. Therefore $o^\ast(a) > o^\ast(b)$ and by Axiom \ref{RCconj}, $\langle o^\ast(a) \rangle \varphi  \land \langle o^\ast(b) \rangle \psi \vdash_{\RC} \langle o^\ast(a) \rangle  \big( \varphi \land \langle o^\ast(b) \rangle \psi\big)$.
\end{proof}

With the following definition we fix a way of translating $\mathbb{F}_{\Gamma_0}$-formulas into formulas in $\form$. However, since different words in $\wormsBC$ might denote the same ordinal, we need a normal form theorem for $\wormsBC$. 

\begin{Definition}
We define ${\sf NF} \subset \wormsBC$ to be the smallest set of $\wormsBC$-words such that $\top \in {\sf NF}$ and for any $\ob a \cb b \in \wormsBC$, if $a, \, b \in {\sf NF}$ and $\big(\ob a \cb b \big)^\ast\in {\sf BNF}$, then $\ob a \cb b \in {\sf NF}$.
\end{Definition} 

Every element of $\wormsBC$ has a unique normal form, as shown by L. Beklemishev in \cite{Beklemishev:2005:VeblenInGLP}.

\begin{Theorem}[Beklemishev]
For each $\al \in \Gamma_0$ we can associate a unique $a_\al \in {\sf NF}$ such that $o^*(a_\al) = \al$.
\end{Theorem}

Now we are ready to translate $\mathbb{F}_{\Gamma_0}$-formulas into $\form$-formulas.

\begin{Definition}\label{defIota}
We define a translation $\iota$ between $\mathbb{F}_{\Gamma_0}$ and $\form$, $\iota: \mathbb{F}_{\Gamma_0} \to \form$, as follows:

\begin{multicols}2
\begin{enumerate}
\item $\top^\iota = \top$;
\item $p^\iota = p$;
\item $(\varphi \land \psi)^\iota = (\varphi^\iota \land \psi^\iota)$;
\item $(\langle \al \rangle \, \varphi)^\iota =  \ob a_\al \cb \varphi^\iota$.
\end{enumerate}
\end{multicols}
\end{Definition}

The following remark follows immediately from the definitions of $\tau$ and $\iota$. 

\begin{Remark} \label{remark:iotatau}
For any $\varphi \in \mathbb{F}_{\Gamma_0}$, $(\varphi^\iota)^\tau = \varphi$.
In particular, if $A \in \mathbb{W}_{\Gamma_0} $ is a worm then $A^\iota \in \wormsBC$ and $o^*(A^\iota) = o((A^\iota)^\ast) = o((A^\iota)^\tau) = o(A)$.
\end{Remark}

With the next definition, we extend the nesting ${\sf N}(a)$ of $a\, {\in}\, \wormsBC$ to $\form$-formulas.

\begin{Definition}
For $\varphi \in \form$, we define the nesting of $\varphi$, ${\sf Nt}(\varphi)$, as the maximum number of nested brackets. That is:
\begin{enumerate}
\item ${\sf Nt}(\top) = {\sf Nt}(p) = {\sf N}(\top) $;
\item ${\sf Nt}(\varphi \land \psi) = \max\big({\sf Nt}(\varphi),\, {\sf Nt}(\psi)\big)$;
\item ${\sf Nt}(\ob a \cb \, \varphi) = \max\big({\sf N}( \ob a \cb), \, {\sf Nt}(\varphi) \big) = \max\big({\sf N}(a) + 1, \, {\sf Nt}(\varphi) \big)$. 
\end{enumerate} 
\end{Definition}

The upcoming remark collects a useful observation concerning the nesting ${\sf Nt}(\varphi)$ of a formula $\varphi$ and its subformulas. This fact can be verified by an easy induction.

\begin{Remark} \label{nestingsubformula} For any $\varphi \in \form$ which is either $\top$ or ${\sf Nt}(\varphi) \geq 1$, there is a subformula $a \in \wormsBC$ of $\varphi$ such that ${\sf Nt}(\varphi) = {\sf Nt}(a)$. Moreover, if ${\sf Nt}(\varphi) \geq 1$, there is a subformula $a \in \wormsBC$ of $\varphi$ such that ${\sf Nt}(\varphi) = {\sf Nt}(a) + 1$. 
\end{Remark}

The following lemma relates the derivability in $\RC$ under $\tau$, and the nesting of formulas in $\form$.
 
\begin{Lemma} \label{Lemma:nestingBoundDerivability}
For any $\varphi,\, \psi \in \form$:
\[
\varphi^\tau \vdash_{\RC} \psi^\tau  \ \Longrightarrow \ {\sf Nt}(\varphi) \geq {\sf Nt}(\psi).
\]
\end{Lemma}
\begin{proof}
Suppose that $\varphi^\tau \vdash_{\RC} \psi^\tau $.
If $\mathcal S (\psi^\tau) = \varnothing$ then it is easy to check that ${\sf Nt} (\psi) = 0$ and there is nothing to prove, so assume otherwise.
Then, by Lemma \ref{Lemma:OrdinalsBoundedByDerivability}.\ref{Ord>0}, $ \max  \mathcal S (\varphi^\tau) \geq \max \mathcal S (\psi^\tau)$.
Using Lemma \ref{lemmStarInstance}, let $a \in \wormsBC$ be a subformula of $\varphi$ such that $o^\ast(a) = \max \mathcal{S}(\varphi^\tau)$. Moreover, since $\mathcal S (\psi^\tau) = \varnothing$, then ${\sf Nt} (\psi) \geq 1$. Therefore, with the help of Remark \ref{nestingsubformula} we can consider $b \in \wormsBC$, a subformula of $\psi$ such that ${\sf Nt}(\psi) = {\sf N}(b) + 1$. If we had ${\sf N}(a) < {\sf N}(b)$ then it would follow from Corollary \ref{Lemma:ordinalsboundnesting} that $o^\ast (a) < o^\ast(b)$, contradicting $ \max  \mathcal S (\varphi^\tau) \geq \max \mathcal S (\varphi^\tau)$.
Thus ${\sf N}(a) \geq {\sf N}(b)$ and ${\sf Nt} (\varphi ) \geq {\sf N}(a) + 1 \geq {\sf Nt}(\psi)$, as needed.
\ignore{
By induction on $\psi$. The only interesting case is $\psi := \ob b \cb \chi$. We can distinguish two cases:
\begin{enumerate}
\item ${\sf Nt}(\psi) = {\sf Nt}(\chi)$, and then it follows by the I.H.;
\item ${\sf Nt}(\psi) = {\sf Nt}(b) + 1$. Since $\varphi^\tau \vdash_{\RC} \psi^\tau$, then by Lemma \ref{Lemma:OrdinalsBoundedByDerivability} we get that $\max \mathcal{S}(\varphi^\tau) \geq \max\mathcal{S}(\psi^\tau)$. Using Lemma \ref{lemmStarInstance}, let $a$ be a subformula of $\varphi$ such that $o^\ast(a) = \max \mathcal{S}(\varphi^\tau)$. Thus, $o^\ast(a) \geq o^\ast (b)$ and therefore, by Corollary \ref{Lemma:ordinalsboundnesting}, ${\sf Nt}(a) \geq {\sf Nt}(b)$. Hence, 
\[
{\sf Nt}(\varphi)\geq {\sf Nt}(a) + 1 \geq {\sf Nt}(b) + 1 = {\sf Nt}(\psi).
\]   
\end{enumerate}
}
\end{proof}

With the following theorem we conclude the proof of the preservability between $\BC$ and $\RC$.

\begin{Theorem}\label{theoPreserve}
For any $\varphi,\, \psi \in \form$:
\[
\varphi^\tau \vdash_{\RC} \psi^\tau  \ \Longleftrightarrow \ \varphi \vdash_\BC \psi.
\]
\end{Theorem}
\begin{proof}
The right-to-left direction is given by Lemma \ref{lemmBCtoRC}, so we focus on the other.
Proceed by induction on ${\sf Nt}(\varphi)$. For the base case, assume ${\sf Nt}(\varphi) = 0$ and $\varphi^\tau \vdash_{\RC} \psi^\tau$. By a subsidiary induction on the length of the derivation of $\varphi^\tau \vdash_{\RC} \psi^\tau$, we set to prove $\varphi \vdash_\BC \psi$. If the derivation has length one it suffices to check $\RC$-Axioms \ref{RCax} and \ref{RCconjel}, which is immediate. If it has length greater than one it must end in a rule. The case for $\RC$-Rule \ref{Rr:1} follows by the I.H.. For $\RC$-Rule \ref{Rr:2}, we have that there is $\chi \in \mathbb{F}_{\Gamma_0}$ such that $\varphi^\tau \vdash_{\RC} \chi$ and $\chi \vdash_{\RC} \psi^\tau$. By Remark \ref{remark:iotatau} and Lemma \ref{Lemma:nestingBoundDerivability}, we get that $\varphi^\tau \vdash_{\RC} (\chi^\iota)^\tau$ and $(\chi^\iota)^\tau \vdash_{\RC} \psi^\tau$ with ${\sf Nt}(\chi^\iota) = 0$.  Thus, by the subsidiary I.H., $\varphi \vdash_\BC \chi^\iota$ and $\chi^\iota \vdash_\BC \psi$ and by $\BC$-Rule \ref{r:2}, $\varphi \vdash_\BC \psi$.

For the inductive step, let ${\sf Nt}(\varphi) = n + 1$. We proceed by a subsidiary induction on the length of the derivation. If $\varphi^\tau \vdash_{\RC} \psi^\tau$ is obtained by means of $\RC$-Axioms \ref{RCax} and \ref{RCconjel}, then clearly $\varphi \vdash_\BC \psi$. If  $\varphi^\tau \vdash_{\RC} \psi^\tau$ is an instance of $\RC$-Axiom \ref{RCtrans}, then we have that $\varphi := \ob a \cb \ob b \cb \chi$ and $\psi := \ob c \cb \chi$ for some $\chi \in \form$ and $a, \, b, \, c \in \wormsBC$ such that $o^\ast(a) = o^\ast(b) = o^\ast(c)$. Hence, $a^\ast \vdash_{\RC} b^\ast$ and $b^\ast \vdash_{\RC} c^\ast$. Since ${\sf Nt}(w) < n + 1$ for $w \in \{a, \, b, \, c \}$, by the main I.H. we have that $a \vdash_\BC b$ and $b \vdash_\BC c$. Thus, we have the following $\BC$-derivation:
\begin{prooftree}
\AxiomC{$\chi \vdash \chi$}
\AxiomC{$b \vdash c$}
\RightLabel{\scriptsize(Rule \ref{r:3})}
\BinaryInfC{$\ob b \cb \chi \vdash \ob c \cb \chi$}
\RightLabel{\scriptsize(Rule \ref{r:3})}
\AxiomC{$a \vdash b$}
\BinaryInfC{$\ob a \cb \ob b \cb \chi \vdash \ob b \cb \ob c \cb \chi$}
\RightLabel{\scriptsize(Rule \ref{r:3})}
\AxiomC{$\chi \vdash \chi$}
\AxiomC{$b \vdash c$}
\BinaryInfC{$\ob b \cb \ob c \cb \chi \vdash \ob c \cb \chi$}
\RightLabel{\scriptsize(Rule \ref{r:2})}
\BinaryInfC{$\ob a \cb \ob b \cb \chi \vdash \ob c \cb \chi$}
\end{prooftree}
If  $\varphi^\tau \vdash_{\RC} \psi^\tau$ is obtained by using $\RC$-Axiom \ref{RCmon}, then $\varphi := \ob a \cb \chi$ and $\psi := \ob b \cb \chi$ for some $\chi \in \form$ and $a, \, b,  \in \wormsBC$ with $o^\ast(a) > o^\ast(b)$. Therefore $a^\ast \vdash_{\RC} \la 0 \ra b^\ast$ and since $\varphi := \ob a \cb \chi$, we have that ${\sf Nt}(a) < {\sf Nt}(\varphi)$. Thus, by the main I.H. $a \vdash_\BC \ob \cb b$ and by $\BC$-Rule \ref{r:3}, $\ob a \cb \chi \vdash_\BC \ob b \cb \chi$. If  $\varphi^\tau \vdash_{\RC} \psi^\tau$ is an instance of $\RC$-Axiom \ref{RCconj}, then we have that $\varphi := \ob a \cb \chi_0 \, \land \, \ob b \cb \chi_1$ and $\psi := \ob a \cb \big(\, \chi_0 \, \land \, \ob b \cb \chi_1\, \big)$, for some $\chi_0, \, \chi_1 \in \form$ and $a, \, b \in \wormsBC$ with $o^\ast(a) > o^\ast(b)$. Since ${\sf Nt}(a) < {\sf Nt}(\varphi)$ and $a^\ast \vdash_{\RC} \la 0 \ra b^\ast$, by the main I.H. we obtain that $a \vdash_\BC \ob \cb b$ and by applying $\BC$-Rule \ref{r:5}, $\ob a \cb \chi_0 \, \land \, \ob b \cb \chi_1 \vdash \ob a \cb \big( \, \chi_0  \, \land \, \ob b \cb \chi_1\, \big)$. Regarding rules, $\RC$-Rule \ref{Rr:1} is immediate and $\RC$-Rule \ref{Rr:3} follows an analogous reasoning to that of Axiom \ref{RCmon}. This way, we only check $\RC$-Rule \ref{Rr:2}. Assume $\varphi^\tau \vdash_{\RC} \psi^\tau$ is obtained by an application of $\RC$-Rule \ref{Rr:2}. Then, there is $\chi \in \mathbb{F}_{\Gamma_0}$ such that $\varphi^\tau \vdash_{\RC} \chi$ and $\chi \vdash_{\RC} \psi^\tau$. By Remark \ref{remark:iotatau} together with Lemma \ref{Lemma:nestingBoundDerivability} we obtain that $\varphi^\tau \vdash_{\RC} (\chi^\iota)^\tau$ and $(\chi^\iota)^\tau \vdash_{\RC} \psi^\tau$ with ${\sf Nt}(\chi) \leq n+1$. By the subsidiary I.H. $\varphi \vdash_\BC \chi^\iota$ and $\chi^\iota \vdash_\BC \psi$ and hence, by $\BC$-Rule \ref{r:2}, $\varphi \vdash_\BC \psi$.
\end{proof}

With this we obtain our main result: an autonomous calculus for representing ordinals below $\Gamma_0$.

\begin{Theorem}
$({\sf NF},\unlhd)$ is a well-order of order-type $\Gamma_0$.
\end{Theorem}

\proof
By Theorem \ref{theoPreserve}, $a \mathrel \lhd b$ if and only if $ b^\tau \vdash_{\mathbf{RC}_{\Gamma_0}} \la 0 \ra a^\tau $ if and only if $o^*(a) < o^*(b)$.
Moreover if $\xi< o^*(a)$ then since $o^*(a) = o(a^\tau)$ and $o$ is the order-type function on $\sf BNF$, by item \ref{itOUniqueTwo} of Lemma \ref{LemmOUnique} there is some $B \wle{} a^\tau$ such that $\xi = o(B)$, hence in view of Remark \ref{remark:iotatau}, $\xi = o^* (B^\iota)$.
Thus by Lemma \ref{LemmOUnique}, $o^*$ is the order-type function on $\sf NF$.
That the range of $o^*$ is $\Gamma_0$ follows from Proposition \ref{proposition:nestedn} which tells us that $o^*(a) < h({\sf N}(a)+1) < \Gamma_0$ for all $a \in \wormsBC$, while if we define recursively $a_0 = \top$ and $a_{n+1} = \ob a_n\cb$, Theorem \ref{TheoTranOrder} and an easy induction readily yield $\Gamma_0 = \lim_{n\to \infty} h(n) = \lim_{n\to \infty} o^*(a_n)$.

It remains to check that $\unlhd $ is antisymmetric.
If $a,b\in {\sf NF}$ and $a\equiv_{\BC} b$, then $a^\tau \equiv_{\RC} b^\tau $ and are in $\sf BNF$.
By Theorem \ref{TheoBNF}, $a^\tau = b^\tau $.
Writing $a = \ob a_1\cb \ldots \ob a_n\cb$ and $b = \ob b_1\cb \ldots \ob b_m\cb$, it follows that $n=m$ and that for $0<i\leq n$, $o^*(a_i) = o^*(b_i)$.
It follows that $a_i^\tau \equiv_{\RC} b_i^\tau$, so that $a_i \equiv_\BC b_i$.
Induction on nesting depth yields $a_i = b_i$, hence $a=b$.
\endproof

\section{The Bracket Principle}

In this section, we adapt Beklemishev's system of fundamental sequences for worms \cite{Beklemishev:2004:ProvabilityAlgebrasAndOrdinals} to elements of $\wormsBC$.
These are sequences $(a\{n\})_{n\in\mathbb N}$ which converge to $a$ with respect to the $\lhd$ ordering whenever $a$ is a limit worm.
Beklemishev uses these fundamental sequences to present a combinatorial statement independent of $\sf PA$.
As we will see, these fundamental sequences generalize smoothly to $\wormsBC$ and provide an independent statement for the theory $\atr$ of Arithmetical Transfinite Recursion, which we briefly recall in the next section. 

Let $a \in \wormsBC$ and write $a=\ob a_1\cb \ldots \ob a_m\cb$, with $m\geq 0$. Then, consider the following cases for arbitrary $n$.
\begin{enumerate}

\item $\top \{n\} = \top$.

\item $a \{n\} = \ob a_2\cb \ldots \ob a_k\cb$ if $a_1=\top$.

\item If $\min (a) \neq a_1$, let $\ell $ be least such that $a_\ell \mathrel\lhd a_1$, otherwise let $\ell = k+1$.
Let $b=\ob a_1 \{n\}  \cb \ldots \ob a_{\ell-1}\cb $ and $ c = \ob a_{\ell }\cb \ldots \ob a_k\cb$ (with $c$ possibly empty).
Then,
$a\{n\} = b^{n+1} c$, where $b^n$ is inductively defined as $\top$ for $n = 0$ and $b\, b^{n'}$ for $n = n' + 1$. 
\end{enumerate}

\begin{Remark}
Beklemishev's definition of fundamental sequences is almost identical, except that each $a_i$ is a natural number and $a_1 \{n\}$ is replaced by $ a_1 - 1 $ when $a_1>0$.
Note, however, that $(\ob \cb^{k+1} )\{n\} = \ob \cb^{k } $, so the two definitions coincide if we represent natural numbers as elements of $\wormsBC$.
\end{Remark}

\begin{Proposition}
If $a\neq \top $, then $a \{n\} \mathrel \lhd a$. \label{AgtAn}
\end{Proposition}
\begin{proof}
We proceed by induction on the nesting of $a$, ${\sf N}(a)$. For ${\sf N}(a) = 1$, we have that $a:= \ob \cb^{k+1} \top$ for some $k < \omega$. Moreover, for any $n$, we have that $a\{n\} = \ob \cb^k \top$. Thus, we have that $a \vdash a$ and $a = \ob \cb a\{n\}$, that is, $a\{n\} \mathrel \lhd a$.

For ${\sf N}(a) = m + 1$, if $a$ is of the form $\ob \cb a'$, we reason as in the previous case. Otherwise, let $a$ be of the form $\ob a_1\cb \ldots \ob a_k\cb$ with $a_1 \neq \top$, then by the I.H. we have that $a_1 \vdash \ob \cb a_1\{ n \}$. Thus,
\begin{center}
\def\fCenter{\ \vdash \ }
\Axiom $\ob a_2\cb \ldots \ob a_k\cb \fCenter \ob a_2\cb \ldots \ob a_k\cb$
\Axiom $a_1 \fCenter \ob \cb a_1\{ n \}$
\def\extraVskip{2.5pt}
\RightLabel{\scriptsize(Rule \ref{r:3})}
\BinaryInf $\ob a_1\cb \ldots \ob a_k\cb \fCenter  \ob a_1\{ n \} \cb \ldots \ob a_k \cb$
\UnaryInf $a \fCenter \ob a_1\{ n \} \cb \ldots \ob a_k \cb$
\DisplayProof
\end{center}
It follows that
\[
a \vdash \ob a_1\{ n \} \cb \ldots \ob a_{\ell-1} \cb \ob a_\ell \cb \ldots \ob a_k\cb.
\]
where $\ell$ is the least such that $a_\ell \mathrel\lhd a_1$ (recall that if $a_1 = \min(a)$, then $\ell = k+1$, and the part $\ob a_\ell \cb \ldots \ob a_k\cb$ is empty). Hence, we get that:
\begin{itemize}
\item $a_1\{n\} \mathrel\lhd a_1$ and
\item $a_1 \mathrel \unlhd a_j$ for $j, \ 1 \, {\leq} \, j \, {<} \, \ell$,
\end{itemize}
and so, $a_1\{n\} \mathrel\lhd a_j$. We can reason as follows:

\begin{prooftree}

\AxiomC{$a \vdash \ob a_1\{ n \} \cb \ldots \ob a_{\ell-1} \cb \ob a_\ell \cb \ldots \ob a_k\cb$}
\AxiomC{$a \vdash \ob a_1 \cb \ldots \ob a_{\ell-1} \cb$}
\def\extraVskip{2.5pt}
\RightLabel{\scriptsize(Rule \ref{r:1})}
\BinaryInfC{$a \vdash \ob a_1\{ n \} \cb \ldots \ob a_{\ell-1} \cb \ob a_\ell \cb \ldots \ob a_k\cb \land \ob a_1 \cb \ldots \ob a_{\ell-1} \cb$}
\end{prooftree}

Let $a' := \ob a_1\{ n \} \cb \ldots \ob a_{\ell-1} \cb \ob a_\ell \cb \ldots \ob a_k\cb$. Combining the fact that $a_j \vdash \ob \cb a_1\{n\}$  for any $j, \ 1 \, {\leq} \, j \, {<} \, \ell$, we have that we can iteratively apply Rule \ref{r:5} and Rule \ref{r:3} to place the modalities  in $\ob a_1 \cb \ldots \ob a_{\ell-1} \cb$ outside the conjunction. More precisely, we start by applying Rule \ref{r:3}, obtaining 
\begin{prooftree}
\AxiomC{$a \vdash \ob a_1 \cb \ldots \ob a_{\ell-1} \cb \land a'$}
\def\extraVskip{2.5pt}
\RightLabel{\scriptsize(Rule \ref{r:3})}
\UnaryInfC{$a \vdash \ob a_1 \cb  \Big( \ob a_2 \cb \ldots \ob a_{\ell-1} \cb  \land a' \Big)$.}
\end{prooftree}
We can observe that since $a_j \vdash \ob \cb a_1\{n\}$ for $j, \ 1 \, {\leq} \, j \, {<} \, \ell$, 
\[
\ob a_2 \cb \ldots \ob a_{\ell-1} \cb  \land a' \vdash \ob a_2 \cb  \Big( \ob a_3 \cb \ldots \ob a_{\ell-1} \cb  \land a' \Big)
\]
These two last inferences can be combined by means of the Rule \ref{r:5} from which we get:
\[
a \vdash \ob a_1 \cb  \ob a_2 \cb  \Big( \ob a_3 \cb \ldots \ob a_{\ell-1} \cb  \land a' \Big).
\]
Iterating this process we get that
\[
a \vdash \ob a_1 \cb \ldots \ob a_{\ell-1} \cb a'
\]
and since, $a_1 \vdash \ob \cb a_1\{n\}$ by Rule \ref{r:3} we get that
\[
a \vdash \ob a_1\{n\} \cb \ldots \ob a_{\ell-1} \cb a'
\]
that is,
\[
a \vdash \ob a_1\{n\} \cb \ldots \ob a_{\ell-1} \cb a' \ob a_1\{ n \} \cb \ldots \ob a_{\ell-1} \cb \ob a_\ell \cb \ldots \ob a_k\cb.
\]
This whole last part of the argument can be iterated, obtaining that 
\[
a \vdash \big( \ob a_1\{ n \} \cb \ldots \ob a_{\ell-1} \cb \big)^{n+1} \ob a_\ell \cb \ldots \ob a_k\cb.
\] 
Thus,

\begin{center}
\def\fCenter{\ \vdash \ }
\Axiom$a \fCenter \big( \ob a_1\{ n \} \cb \ldots \ob a_{\ell-1} \cb \big)^{n+1} \ob a_\ell \cb \ldots \ob a_k\cb \land \ob a_1 \cb$
\def\extraVskip{2.5pt}
\RightLabel{\scriptsize(Rules \ref{r:5} and \ref{r:3})}
\UnaryInf$a \fCenter \ob a_1 \cb \big( \ob a_1\{ n \} \cb \ldots \ob a_{\ell-1} \cb \big)^{n+1} \ob a_\ell \cb \ldots \ob a_k\cb$
\RightLabel{\scriptsize(Rule \ref{r:3})}
\UnaryInf$a \fCenter \ob  \cb \big( \ob a_1\{ n \} \cb \ldots \ob a_{\ell-1} \cb \big)^{n+1} \ob a_\ell \cb \ldots \ob a_k\cb$
\DisplayProof
\end{center}
Therefore, we can conclude that $a\{n\} \mathrel\lhd a$.
\end{proof}

Now, let us define for $a\in \wormsBC$ and $n\in \mathbb N$, a new worm $a\lBrace n\rBrace \in \wormsBC$ recursively by $a\lBrace 0\rBrace= a$ and $a\lBrace n+1\rBrace= a\lBrace n \rBrace \{n+1\}$.

\begin{Theorem}\label{theoBracketPrinciple}
For each $a \in \wormsBC$, there is $ i \in\mathbb N$ such that $a\lBrace i \rBrace= \top$.
\end{Theorem}

\begin{proof}
Assume towards a contradiction that there is no $i \in \mathbb{N}$ such that $a\lBrace i \rBrace = \top$. Then, Proposition \ref{AgtAn} yields $a\lBrace i + 1 \rBrace \mathrel\lhd a\lBrace i \rBrace$ for all $i$, and thus $\big(o^\ast (a\lBrace i \rBrace)\big)_{i\in \mathbb{N}}$ defines an infinite descending chain of ordinals, contradicting the well-foundedness of $\Gamma_0$.
\end{proof}

\section{Independence}

We will now show that Theorem \ref{theoBracketPrinciple} is independent of $\atr$.
The theory $\atr$ is defined in the language of second order arithmetic, which extends the language of $\sf PA$ with a new sort of variables $X,Y,Z$ for sets of natural numbers with atomic formulas $t\in X$ and quantification $\forall X \varphi(X)$ ranging over sets of natural numbers.
Using coding techniques, the language of arithmetic suffices to formalize many familiar mathematical notions such as real numbers and continuous functions on the real line.
More relevant to us, (countable) well-orders and Turing jumps can be formalized in this context.
The particulars are not important for our purposes, but these notions are treated in detail in \cite{Simpson:2009:SubsystemsOfSecondOrderArithmetic}.

The theory $\mathbf{ACA}_0$ is the second order analogue of Peano Arithmetic, defined by extending Robinson's $\sf Q$ with the induction axiom stating that every non-empty set has a least element and the axiom scheme stating that $\{x\in \mathbb N : \varphi (x)\}$ is a set, where $\varphi $ does not contain second order quantifiers but possibly contains set-variables.
Equivalently, $\mathbf{ACA}_0$ can be defined with an axiom that states that for every set $X$, the Turing jump of $X$ exists.
The theory of Arithmetical Transfinite Recursion, $\atr$, can then be obtained by extending $\mathbf{ACA}_0$ with an axiom stating that for every set $X $ and well-order $\alpha$, the $\alpha^{\rm th}$ Turing jump of $X$ exists.
This theory is related to predicative mathematics \cite{Feferman:1964:SystemsOfPredicativeAnalysis}, and discussed in detail in \cite{Simpson:2009:SubsystemsOfSecondOrderArithmetic}.

The proof-theoretic ordinal of $\atr$ is $\Gamma_0$.
In order to make this precise, we need to study Veblen hierarchies in some more detail; recall that we have defined them in Section \ref{secHyper}.
The Veblen normal form of $\xi>0$ is the unique expression $\upphi_\alpha \beta + \gamma = \xi$ such that $\gamma<\xi $ and $\beta <\upphi_\alpha \beta$. We will call this the {\em Veblen normal form} of $\xi$ and write $\xi \vnf \upphi_\alpha \beta + \gamma$.
The order relation between elements of $\Gamma_0$ can be computed recursively on their Veblen normal form.
Below we consider only $\xi,\zeta>0$, as clearly $0 < \ot_\alpha\beta + \gamma$ regardless of $\alpha,\beta,\gamma$.
The following is found in e.g.~\cite{Pohlers:2009:PTBook}.

\begin{Lemma}
Given $\xi,\xi'<\Gamma_0$ with $\xi=\ot_{\alpha}{\beta }+{\gamma }$ and $\xi' = \ot_{\alpha'}{\beta'}+{\gamma'}$ both in Veblen normal form, $\xi<\xi'$ if and only if one of the following holds:
\begin{multicols}2
\begin{enumerate}

\item $\alpha=\alpha'$, $\beta=\beta'$ and $\gamma<\gamma'$;

\item $\al<\al'$ and $\be<\upphi_ {\alpha'}\beta'$;

\item $\al=\al'$ and $\be<\be'$, or

\item $\al'<\al$ and $\upphi _\al\be < \be'$.

\end{enumerate}
\end{multicols}
\end{Lemma}

In order to prove independence from $\atr$, we also need to review fundamental sequences based on Veblen notation.

\begin{Definition} \label{fundseq}
For $\xi < \Gamma_0$ and $x<\omega$, define $\fs \al x$ recursively as follows.
First we set $ \xcases x \alpha = x + 1$ if $\alpha$ is a successor, $ \xcases x\alpha = 1$ otherwise.
Then, define:
\begin{multicols}2
\begin{enumerate}

\item $\fs 0 x = 0$.

\item $\fs{ ( \ot_\alpha \be + \gamma ) }x= \ot_\alpha \be +  \fs \ga x$ if $\ga > 0$.

\item $\fs{(\upphi_00)} x =0$ (note that $  {\upphi_00} = 1$).

\item $\fs{\upphi_0(\be+1)}x=\upphi_0\be\cdot (x+2)$.

\columnbreak

\item $\fs{(\upphi_\alpha0)}x:=\upphi^{\xcases x\alpha}_{\fs \alpha x }  0 $ if $\alpha > 0$.

\item $\fs{\ot_\alpha(\be+1)}x:=
\ot^{\xcases x\alpha}_{\fs\alpha x} (\ot_\alpha\be + 1 )$ if $\alpha > 0 $. \label{fssuc}

\item $\fs{(\ot_\alpha\lambda)}x:=\ot_{\alpha}(\fs \lambda x)$ if $\lambda$ is a limit. \label{fslim}

\end{enumerate}
\end{multicols}
\end{Definition}

For an ordinal $\xi<\Gamma_0$ and $n\geq 0$ we define inductively $\fsi \xi 0 =  \xi $ and $\fsi\xi{ n+1} = \fs{\fsi \xi  n }{n+1}$.
The system of fundamental sequences satisfies the {\em Bachmann property} \cite{Schmidt77}:

\begin{Proposition}\label{propBachmann}
If $\alpha,\beta < \Gamma_0$ and $k< \omega$ satisfy $\fs \alpha k < \beta < \alpha$, then $\fs\alpha k \leq \fs \be 1$.
\end{Proposition}

This property is useful because it allows us to appeal to the following, proven in \cite{FernWierCiE2020}.

\begin{Proposition}\label{propBachmannMaj}
If $(\alpha_i)_{i \in \mathbb N}$ is a sequence of elements of $\Gamma_0$ such that for all $i$,
\[\alpha_i [i+1] \leq \alpha_{i+1} \leq \alpha_i,\]
it follows that for all $i$, $\alpha_i \geq \fsi {\alpha_0} {i}$.
\end{Proposition}

This allows us to establish new independence results by appealing to the fact that $\atr$ does not prove that the process of stepping down the fundamental sequences below $\Gamma_0$ always reaches zero.
By recursion on $n$ define $\gamma_n<\Gamma_0$ as follows, $\gamma_0:=0$ and $\gamma_{n+1}:=\upphi_{\ga_n} 0 .$
Then, $ \forall m \exists \ell \ ( \fsi{\ga_m} \ell  =0 )$ is not provable in $ \atr $.
In fact, the least $\ell$ such that $\fsi{\ga_m} \ell  =0$ grows more quickly than any provably total computable function.

Let us make this precise.
For our purposes, a partial function $f\colon \mathbb N \to \mathbb N$ is {\em computable} if there is a $\Sigma_1$ formula $\varphi_f (x,y)$ in the language of first order arithmetic (with no other free variables) such that for every $m,n$, $f(m) = n$ if and only if $\varphi_f (m,n)$ holds.
The function $f$ is {\em provably total} in a theory $T$ if $T\vdash \forall x \exists y \varphi _f (x,y) $ (more precisely, $f$ is provably total if there is at least one such choice of $ \varphi _f $).
Then, the function $F$ such that $F(m)$ is the least $\ell$ with $\fsi{\ga_m} \ell  =0 $ is computable and, by the well-foundedness of $\Gamma_0$, total.
Moreover, it gives an upper bound for all the provably total computable functions in $\atr$.

\begin{Theorem}[\cite{Schmidt77}]\label{theoATRInd}
If $f \colon \mathbb N \to \mathbb N$ is a computable function that is provably total in $\atr$, then $\exists N \forall n>N \ (f(n) < F(n))$.
\end{Theorem}

So, our goal will be to show that the witnesses for Theorem \ref{theoBracketPrinciple} grow at least as quickly as $F$, from which we obtain that the theorem is unprovable in $\atr$.
We begin with a straightforward technical lemma, which will be useful below.

\begin{Lemma} \label{monBracketWorm}
Let $a_0, \ldots ,a_k, b_0, \ldots b_k \in \wormsBC$. If for any $j$ with $ 0 \leq j \leq k$, we have that $a_j \mathrel \unlhd b_j$, then $\ob a_0 \cb \ldots \ob a_k \cb \mathrel\unlhd   \ob b_0 \cb \ldots \ob b_k \cb$.
\end{Lemma}
\begin{proof}
By a  simple induction on $k$ applying Rule \ref{r:3}. For the base, we have that
\begin{prooftree}
\AxiomC{$\top \vdash \top$}
\AxiomC{$a_0 \mathrel \unlhd b_0$}
\def\extraVskip{2.5pt}
\RightLabel{\scriptsize(Rule \ref{r:3})}
\BinaryInfC{$\ob b_0 \cb \vdash \ob a_0 \cb$}
\end{prooftree}
The inductive step follows from the I.H. with an analogous reasoning.
\end{proof}

It will be useful to extend the $\uparrow$ operation of Definition \ref{defUparrow} to elements of $\wormsBC$ using the $\iota$ operation of Definition \ref{defIota}.

\begin{Definition} \label{Def:uparrow}
Let $a \in \wormsBC$, $\al < \Gamma_0$. By $\al \uparrow a$ we denote the expression $(\al \uparrow a^\ast)^\iota \in \wormsBC$.
\end{Definition}

\begin{Lemma} \label{sequparrow}
For any $a \in \wormsBC$, $k < \omega$ and $\al < \Gamma_0$:
\[
o^\ast \Big( (\al \uparrow a)\{k\} \Big) \geq o^\ast \Big( \al \uparrow (a\{k\})\Big).
\]
\end{Lemma}

\begin{proof}
By induction on the complexity of $a$ with the base case being trivial. For the inductive step, let $a := \ob a_0 \cb \ob b_0 \cb \ldots \ob b_j \cb$. Observe that 
\[
\al \uparrow a = \ob  \al \uparrow a_0 \cb  \ob \al \uparrow b_0 \cb \ldots \ob \al \uparrow  b_j  \cb
\]
and so,
\[
(\al \uparrow a)\{k\} = \big( \ob  (\al \uparrow a_0)\{k\} \cb  \ob \al \uparrow b_0 \cb \ldots \ob \al \uparrow  b_i \cb \big)^{k+1} \ob \al \uparrow  b_{i+1}  \cb \ldots \ob \al \uparrow  b_j  \cb
\]
By the I.H. we have that $o^\ast \big( (\al \uparrow a_0)\{k\} \big) \geq o^\ast \big( \al \uparrow (a_0\{k\}) \big)$ and so, $\al \uparrow (a_0\{k\}) \mathrel \unlhd (\al \uparrow a_0)\{k\}$. Therefore, applying Lemma \ref{monBracketWorm} we can conclude that:
\[
\big( \ob  \al \uparrow (a_0 \{k\} )\cb  \ob \al \uparrow b_0 \cb \ldots \ob \al \uparrow  b_i \cb \big)^{k+1} \ob \al \uparrow  b_{i+1}  \cb \ldots \ob \al \uparrow  b_j  \cb \mathrel \unlhd  (\al \uparrow a)\{k\} 
\]
Thus with the help of Theorem \ref{TheoTranOrder} and Proposition \ref{propEvsPhi} together with Definition \ref{Def:uparrow},
\begin{tabbing}
$o^\ast \big( (\al \uparrow a)\{k\} \big)$ \= ${\geq}$ \\[0.20cm]
\> \hspace{-0.3cm}$o^\ast \Big( \big( \ob  \al \uparrow (a_0 \{k\} )\cb  \ob \al \uparrow b_0 \cb \ldots \ob \al \uparrow  b_i \cb \big)^{k+1} \ob \al \uparrow  b_{i+1}  \cb \ldots \ob \al \uparrow  b_j  \cb \Big)$\\[0.20cm]
\> $= e^\al \Big( o^\ast \big( ( \ob  a_0 \{k\} \cb  \ob b_0 \cb \ldots \ob b_i \cb \big)^{k+1} \ob b_{i+1}  \cb \ldots \ob b_j  \cb \big) \Big)$\\[0.20cm]
\> $= e^\al \big( o^\ast ( a\{k\} ) \big)$\\[0.20cm]
\> $= o^\ast \big( \al \uparrow (a\{k\}) \big)$. 
\end{tabbing}
 
\end{proof}

The following is the key lemma in showing that Theorem \ref{theoBracketPrinciple} implies that stepping down the fundamental sequences eventually reaches zero.

\begin{Lemma}\label{lemmBCvsFS}
If $ a \in \wormsBC$, $a' = \ob \cb a$ and $1<k<\omega$ then
\[ \fsi{ o^*(a)}k \leq o^*( a' \lBrace k + 1\rBrace).\]
\end{Lemma}

\proof
It suffices to show that $o^*(a) [k] \leq o^*(a\{k + 1\})$ for every $a \in \wormsBC$.
The lemma then follows from the Bachmann property (Proposition \ref{propBachmann}), since then we can apply Proposition \ref{propBachmannMaj} to the sequence $(\alpha_i)_{i\leq \omega} $ with $\alpha_0 = o^*(a)$ and $\alpha_{i+1} = a' \lBrace i + 2\rBrace$ to obtain $\fsi{ o^*(a ) } k < o^* (a' \lBrace k + 1 \rBrace )$.
We proceed by induction on $o^\ast(a)$.
If $o^\ast(a) = 0$, the claim is trivially true, so we assume otherwise. Write $o^\ast(a) = \upphi_\delta \beta +\gamma$ in Veblen normal form and consider the following cases.
\medskip

\begin{enumerate}[label*={\sc Case \arabic*},wide, labelwidth=!, labelindent=0pt]

\item ($\gamma>0$). Then, $o^\ast(a)$ is additively decomposable.
By inspection on Theorem \ref{TheoTranOrder}, $a  $ is of the form $ a_0 \ob \cb b$ and $o^\ast (a)[k] = o^\ast(b) + o^\ast(a_0)[k]$. By the I.H., $o^\ast(a_0)[k] \leq o^\ast(a_0\{k+1\})$, therefore $o^\ast(b) + o^\ast(a_0)[k] \leq o^\ast(b) + o^\ast(a_0\{k+1\}) = o^\ast( a_0\{k+1\} \ob \cb b) = o^\ast(a\{k+1\})$. 
\medskip

\item ($\gamma=0$).
Then, $ o^\ast(a ) = \varphi_\delta \beta$, with $\delta, \beta < o^\ast(a )$. We distinguish several sub-cases.
\medskip

\begin{enumerate}[label*= .\arabic*,wide, labelwidth=!, labelindent=0pt]

\item ($\delta = 0$ and $\beta = \beta'+1$). Then, inspection of Theorem \ref{TheoTranOrder} and Proposition \ref{propEvsPhi} (which we will no longer mention in subsequent cases) shows that \[
o^\ast(a) = \upphi_0 \beta' + 1 = e^1 (\beta' + 1) = e^1 \big( o^\ast(\ob\cb b)\big)
\]
where $o^\ast(b) = \be'$. Thus, $a = \ob \ob \cb \cb ( 1 \uparrow b)$. We can observe that $a\{k+2\} =  \Big( \ob \cb\ \big( 1 \uparrow b \big) \Big)^{k+2}$ and so we have that:
\begin{align*}
o^\ast ( a\{k+1\}) & = o^\ast (1 \uparrow b ) \cdot (k+2) + 1 = \upphi_0 \big( o^\ast (b) \big) \cdot (k+2) + 1\\ 
&> \upphi_0 (\be') \cdot (k+2)
= \upphi_0 (\be' + 1 ) [k].
\end{align*}

\item ($\delta>0$ and $\beta = 0$). Let $d$ be such that $o^\ast (d) = \delta $. Then, we have that $a = \ob b \cb$ with 
\[o^\ast (b) = \omega^\delta = e^1 (\delta) = e^1 (o^\ast(d))= o^\ast(1 \uparrow d).\]
Therefore, $a = \ob (1\uparrow d) \cb$, and so we have that
\begin{align*}
o^\ast (a\{k + 1\}) & = o^\ast\big( \ob (1 \uparrow d)\{k + 1\}\cb^{k+2} \big)\\
&= e^{o^\ast\big((1 \uparrow d)\{k+1\}\big)} (k+2) \geq e^{\omega^\delta [k]} (k+2),
\end{align*}
where the last inequality uses the induction hypothesis on $\delta < o^* (a)$.
We claim that
\[
e^{\omega^\delta [k]} (k+2) \geq \upphi_{\delta[k]}^{\xcases k\delta} (1);
\]
indeed, if $ \delta = \delta'+1$ is a successor, then $\omega^\delta [k] = \omega^{\delta'}\cdot ( k +2)$, so
\[e^{\omega^\delta [k]} (k+2) = e^{\omega^{\delta'}\cdot (k+2)} (k+2) \geq \upphi_{\delta'}^{k+2} (k+1)  > \upphi_{\delta[k]}^{\xcases k\delta} (1), \]
while, if $\delta$ is a limit,
\[e^{\omega^\delta [k]} (k+2) = e^{\omega^{\delta[k]} } (k+2) \geq \upphi_{\delta[k]}  (k+1) \geq \upphi_{\delta[k]}^{\xcases k\delta} (1). \]
Hence, $o^\ast (a) [k] \leq o^\ast (a\{k + 1\})$.
\medskip

\item ($\delta>0$ and $\be=\be'+1$).
By Definition \ref{fundseq}, Item \ref{fssuc}, we have that $o^\ast(a)[k] = \upphi_{\delta[k]}^{\xcases k\delta} \big( \upphi_\delta (\be') + 1 \big)$. 
%
Since $\upphi_\delta (\be') $ is infinite, we have that $ 1+ \upphi_\delta (\be')  = \upphi_\delta (\be') $, and hence $\upphi_{\delta[k]}^ i \big( \upphi_\delta (\be') + 1 \big) = e^{\omega^{\delta[k]} \cdot i} \big ( e^{\omega^\delta} (\be') + 1 \big) $ for all $i$.
Then, $a  = \ob \ob d \cb \cb \omega^\delta \uparrow b$, where $o^\ast(b) = 1 + \be'$ and $o^\ast(d) = \delta$, and we have that
\[
o^\ast(a)[k] = o^\ast\big( \omega^{\delta[k]} \uparrow \ob \cb ( \omega^\delta \uparrow b ) \big).
\]
By the I.H. and Lemma \ref{monBracketWorm},
\begin{align*}
o^\ast\big( \omega^{\delta[k]} \uparrow \ob \cb ( \omega^\delta \uparrow b ) \big) & \leq  o^\ast\big( o^*(d\{k+ 1 \} ) \uparrow \ob \cb ( \omega^\delta \uparrow b ) \big) \\
& = o^\ast\big(  \ob \ob d\{k+ 1\} \cb \cb ( \omega^\delta \uparrow b )  \big) \leq o^\ast (  a\{k + 1\} ). 
\end{align*}

\item ($\delta>0$ and $\beta \in \Lim$).
Then, $a = \omega^\delta \uparrow b$, where $o^\ast (b) = \beta$. By Definition \ref{fundseq}, Item \ref{fslim}, we get that $\upphi_\delta (\beta)[k] = \upphi_\delta \big(o^\ast(b)[k]\big)$ and since by the I.H. $o^\ast(b)[k] \leq o^\ast (b\{k + 1\})$, we have that $\upphi_\delta \big(o^\ast(b)[k]\big) \leq \upphi_\delta \big(o^\ast(b\{k + 1\})\big)$. On the other hand, 
\[
\upphi_\delta \big(o^\ast(b\{k + 1\})\big) = e^{\omega^\delta} \big(o^\ast(b\{k + 1\})\big) = o^\ast \Big( \omega^\delta \uparrow (b\{k + 1\}) \Big).
\] 
By Lemma \ref{sequparrow}, $o^\ast \big( \omega^\delta \uparrow (b\{k + 1 \}) \big) \leq o^\ast \big( (\omega^\delta \uparrow b)\{k + 1 \} \big)$. Thus, $o^\ast(a)[k] \leq o^\ast(a\{k + 1\})$.
\end{enumerate}
\end{enumerate}
\endproof

\begin{Theorem}\label{theoBCInd}
Theorem \ref{theoBracketPrinciple} is not provable in $\atr$.
\end{Theorem}

\begin{proof}
Let $\gamma_m$ be as in Theorem \ref{theoATRInd} and define recursively $a_0:=\top$, $a_1 := \ob \cb$, and $a_{n+2}:=\ob \ob \ob a_n \cb \cb \cb.$
Then, set $a'_n = \ob\cb a_n$.
We can observe that $o^\ast (a_m)  = \gamma_m $. For $m < 2$ the claim is trivial. For $m \geq 2$, it follows from a simple induction on $m$ together with the fact that $o^\ast(a^m) = h(m+1)$, as given in Definition \ref{defHFun}. Thus, Lemma \ref{lemmBCvsFS} yields
\[\fsi {\gamma_m} k = \fsi {o^\ast(a_m)} k  \leq o^\ast(  a_m' \lBrace k + 1 \rBrace ) \]
for all $k$.
It follows that, if $  a_m' \lBrace k + 1  \rBrace   = \top $ for some $k$, then $\fsi {\gamma_m} k = \fsi {o^\ast(a_m)} k =   0 $ for some $k$.

Recall that we had defined $F(m)$ to be the least $\ell$ so that $ \fsi {\gamma_m} \ell = 0$.
Defining $G(m)$ to be the least $k$ such that $ a_m' \lBrace k + 1  \rBrace = \top$, it follows that $F(m) \leq G(m)$ for all $m$, and hence, by Theorem \ref{theoATRInd}, $\atr$ does not prove that $G(m)$ is total; in other words, $\atr \not \vdash \forall m \exists k \ ( a'_m \lBrace k + 1 \rBrace = \top )$, and Theorem \ref{theoBracketPrinciple} is unprovable in $\atr$.
\end{proof}

\section{Concluding remarks}

Beklemishev's `brackets' provided an autonomous notation system for $\Gamma_0$ based on worms, but did not provide a method for comparing different worms without first translating into a more traditional notation system.
Our calculus $\BC$ shows that this is not necessary, and indeed all derivations may be carried out entirely within the brackets notation.
To the best of our knowledge, this yields the first ordinal notation system presented as a purely modal deductive system.

Our analysis is purely syntactical and leaves room for a semantical treatment of $\BC$.
As before one may first map $\BC$ into $\RC $ and then use the Kripke semantics presented in \cite{Beklemishev:2013:PositiveProvabilityLogic,Dashkov:2012:PositiveFragment}, but we leave the question of whether it is possible to define natural semantics that work only with $\BC$ expressions and do not directly reference ordinals.

The independence of Theorem \ref{theoBracketPrinciple} provides a relatively simple combinatorial statement independent of the rather powerful theory $\atr$.
In particular, our fundamental sequences for worms enjoy a more uniform definition than those based on Veblen functions.
It is of interest to explore whether this can be extended to provide statements independent of much stronger theories.
In \cite{Spiders}, we suggest variants of the brackets notation for representing the Bachmann-Howard ordinal and beyond.
Sound and complete calculi for these systems remain to be found, as do natural fundamental sequences leading to new independent combinatorial principles.

%
\bibliographystyle{plain}
 \bibliography{References}

\end{document}